\documentclass[onecolumn,amsthm,letterpaper,draft]{autart}

\usepackage{amsfonts}
\usepackage{amsmath}
\usepackage{amsthm}
\usepackage{graphicx}

\usepackage[numbers,sort&compress,longnamesfirst]{natbib} 
\graphicspath{{graphics_included/}}

\newcommand{\R}{\mathbb{R}}

\newcommand{\N}{\mathbb{N}}

\newcommand{\norm}[1]{\lVert{#1}\rVert}
\newcommand{\eemph}[1]{\textbf{\textit{#1}}}

\let\bbl\Bigl

\let\bbbbl\Biggl

\let\bbr\Bigr

\let\bbbbr\Biggr

\begin{document}

\begin{frontmatter}

\title{Exponentially Stable Nonlinear Systems have Polynomial Lyapunov Functions on Bounded Regions}

\thanks[footnoteinfo]{Corresponding author M.~M.~Peet.}

\author[mp]{Matthew M. Peet}\ead{matthew.peet@inria.fr},

\address[mp]{INRIA-Rocquencourt,Domaine de Voluceau, Rocquencourt BP105, 78153 Le Chesnay Cedex, France}


\begin{abstract}
This paper presents a proof that existence of a polynomial Lyapunov
function is necessary and sufficient for exponential stability of
sufficiently smooth nonlinear ordinary differential equations on
bounded sets. The main result states that if there exists an n-times
continuously differentiable Lyapunov function which proves
exponential stability on a bounded subset of $\R^n$, then there
exists a polynomial Lyapunov function which proves exponential
stability on the same region. Such a continuous Lyapunov function
will exist if, for example, the right-hand side of the differential
equation is polynomial or at least $n$-times continuously
differentiable. The proof is based on a generalization of the
Weierstrass approximation theorem to differentiable functions in
several variables. Specifically, we show how to use polynomials to
approximate a differentiable function in the Sobolev norm
$W^{1,\infty}$ to any desired accuracy. We combine this
approximation result with the second-order Taylor series expansion
to find that polynomial Lyapunov functions can approximate
continuous Lyapunov functions arbitrarily well on bounded sets. Our
investigation is motivated by the use of polynomial optimization
algorithms to construct polynomial Lyapunov functions.
\end{abstract}

\end{frontmatter}


\section{Introduction}
The Weierstrass approximation theorem was proven in
1885~\cite{weierstrass_1885}. This result demonstrated that
real-valued polynomial functions can approximate real-valued
continuous functions arbitrarily well with respect to the supremum
norm on a compact interval. Various structural generalizations of
the Weierstrass approximation theorem have focused on generalized
mappings, as in~\citet{stone_1948}, and on alternate topologies, as
in~\citet{krein_1945}. Polynomial approximation of differentiable
functions has been studied in numerical analysis of differential and
partial differential equations. Results of relevance include the
Bramble-Hilbert Lemma~\cite{bramble_1971} and its generalization
in~\citet{dupont_1980}, and Jackson's theorem~\cite{jackson_1911},
generalizations of which can be found in the work
of~\citet{timan_1966}.

The approximation of Sobolev spaces by smooth functions has been
studied in a number of contexts. In~\citet{everitt_1993}, the
density of polynomials of a single variable in certain Sobolev
spaces was discussed. Other work considers weighted Sobolev spaces,
as in~\citet{portilla_xxx}.

The density of infinitely continuously differentiable test functions
in Sobolev spaces has been studied in the context of partial
differential equations. Important results include the Meyers-Serrin
Theorem~\cite{meyers_1964}, extensions of which can be found
in~\citet{adams_book} or~\citet{evans_book}.


Recently, the density of polynomials in the space of continuous
functions has been used to motivate research on optimization over
parameterized sets of positive polynomial functions. Relevant
results include an extension of the Weierstrass approximation
theorem to polynomials subject to affine
constraints~\cite{peet_2007IFAC1}. Examples of algorithms for
optimization over cones of positive polynomials include
SOSTOOLS~\cite{prajna_2004}, which optimizes over the cone of sums
of squares of polynomials, and Gloptipoly~\cite{henrion_2001}, which
optimizes over the dual space to obtain bounds on the primal
objective.

One application of polynomial optimization which has received
particular attention has been construction of polynomial Lyapunov
functions for ordinary nonlinear differential equations of the form
\begin{equation*}
    \dot{x}(t)=f(x(t)),
\end{equation*}
where $f:\R^n \rightarrow \R^n$. Several converse theorems have
shown that local exponential stability of this system implies the
existence of a Lyapunov function with certain properties of
continuity. The reader is referred to~\citet{hahn_1967}
and~\citet{krasovskii_1963} for an extensive treatment of converse
theorems of Lyapunov. It is not generally known, however, under what
conditions an exponentially stable system has a polynomial Lyapunov
function.

The main conclusion of this paper is summarized in
Theorem~\ref{prop:conversepoly}, where we show that if $f$ is
$n$-times continuously differentiable, then exponentially stability
of $f$ on a ball is equivalent to the existence of a polynomial
Lyapunov function which decreases exponentially along trajectories
contained in that ball. The smoothness condition is automatically
satisfied if $f$ is polynomial.

To establish this converse Lyapunov result, we prove two extensions
to the Weierstrass approximation theorem. In
Theorem~\ref{thm:theoremtemp1}, we show that for bounded regions,
polynomials can be used to approximate continuously differentiable
multivariate functions arbitrarily well in a variety of norms,
including the Sobolev norm $W^{1,\infty}$. This means that for any
continuously differentiable function and any $\gamma>0$, we can find
a polynomial which approximates the function with error $\gamma$ and
whose partial derivatives approximate those of the function with
error $\gamma$. The proof is based on a construction using
approximations to the partial derivatives.

Our second extension combines a second order Taylor series expansion
with the Weierstrass approximation theorem to find polynomials which
approximate functions with a pointwise weight on the error given by
\[
w(x)=\frac{1}{x^T x}.
\]

These two extensions are combined into the main polynomial
approximation result, which is Theorem~\ref{thm:theorem1}. The
application to Lyapunov functions is given in
Proposition~\ref{prop:prop1}. Proposition~\ref{prop:prop1} states
that if there exists a sufficiently smooth continuous Lyapunov
function which proves exponential stability on a bounded set, then
there exists a polynomial Lyapunov function which proves exponential
stability on the same set. In Section~\ref{sec:lyapunov}, we
interpret our work as a converse Lyapunov theorem and briefly
discuss implications for the use of polynomial optimization to prove
stability of nonlinear ordinary differential equations.



\section{Notation and Background}

Let $\N^n$ denote the set of length $n$ vectors of non-negative
natural numbers. Denote the unit cube in $\N^n$ by $Z^n := \{\alpha
\in \N^n : \alpha_i \in \{0,1\}\}$. For $x \in \R^n$,
$\norm{x}_\infty = \max_i |x_i|$ and $\norm{x}_2=\sqrt{x^T x}$.
Define the unit cube in $\R^n$ by $B:=\{x \in \R^n \, : \,
\norm{x}_\infty \le 1 \}$. Let $\mathcal{C}(B)$ be the Banach space
of scalar continuous functions defined on $B \subset \R^n$ with norm
\[\norm{f}_\infty :=\sup_{x \in B}\|f(x)\|_\infty.\]
For operators $h_i:X\rightarrow X$, let $\prod_i h_i:X\rightarrow X$
denote the sequential composition of the $h_i$. i.e.
\[
    \prod_i h_i:=h_1 \circ h_{2}\circ \cdots \circ h_{n-1}\circ h_n .
\]
For a sufficiently regular function $f:\R^n \rightarrow \R$ and
$\alpha \in \N^n$, we will often use the following shorthand to
denote the partial derivative
\[
    D^\alpha f(x) := \frac{\partial^\alpha }{\partial x^\alpha}f(x)= \prod_{i=1}^n \frac{\partial^{\alpha_i}}{\partial
    x^{\alpha_i}}f(x),
\]
where naturally, $\partial^0 f/\partial x_i^0=f$. For $\Omega
\subset \R^n$, we define the following sets of differentiable
functions.
\begin{align*}
    &\mathcal{C}_1^i(\Omega):=\bbbbl\{f \, : \, D^\alpha f \in \mathcal{C}(\Omega)\quad
\text{for any $\alpha \in \N^n$ such that }
\norm{\alpha}_1=\sum_{j=1}^n \alpha_j \le i.\bbbbr\}
\end{align*}
\begin{align*}
    &\mathcal{C}_\infty^i(\Omega):=\bbl\{f \, : \, D^\alpha f \in \mathcal{C}(\Omega)\quad
\text{for any $\alpha \in \N^n$ such that }
\norm{\alpha}_\infty=\max_{j} \alpha_j \le i.\bbr\}
\end{align*}

$\mathcal{C}^\infty(B)=\mathcal{C}_1^\infty(B)=\mathcal{C}_\infty^\infty(B)$
is the logical extension to infinitely continuously differentiable
functions. Note that in $n$-dimensions, $\mathcal{C}_1^i(B) \subset
\mathcal{C}_\infty^i(B) \subset \mathcal{C}_1^{i n}(B)$. We will
occasionally refer to the Banach spaces $W^{k,p}(\Omega)$, which
denote the standard Sobolev spaces of locally summable functions
$u:\Omega \rightarrow \R$ with \eemph{weak} derivatives $D^\alpha u
\in L_p(\Omega)$ for $|\alpha|_1 \le k$ and norm
\[
    \norm{u}_{W^{k,p}}:=\sum_{|\alpha|_1 \le k }\left\|D^\alpha
    u\right\|_{L_p}.
\]

%

The following version of the Weierstrass approximation theorem in
multiple variables comes from~\citet{timan_1966}.

\begin{thm}
Suppose $f:\R^n \rightarrow \R$ is continuous and $G \subset \R^n$
is compact. Then there exists a sequence of polynomials which
converges to $f$ uniformly in $G$.
\end{thm}


\section{Approximation of Differentiable Functions}

Most of the technical results of this paper concern a constructive
method of approximating a differentiable function of several
variables using approximations to the partial derivatives of that
function. Specifically, the result states that if one can find
polynomials which approximate the partial derivatives of a given
function to accuracy $\gamma/2^n$, then one can construct a
polynomial which approximates the given function to accuracy
$\gamma$, and all of whose partial derivatives approximate those of
the given function to accuracy $\gamma$.

The difficulty of this problem is a consequence of the fact that,
for a given function, the partial derivatives of that function are
not independent. For example, we have the identities
\[
    D^{(1,1,0)}f(x,y,z)=D^{(1,0,0)}\left(f(x,y,0) + \int_{0}^z D^{(0,0,1)}f(x,y,s)\,ds\right)
    =D^{(1,0,0)}\left(D^{(0,1,0)}f(x,y,z)\right),
\]
among others. Therefore, given approximations to the partial
derivatives of a function, these approximations will, in general,
not be the partial derivatives of any function. Then the problem
becomes, for each partial derivative approximation, how to extract
the information which is unique to that partial derivative in order
to form an approximation to the original function. The following
construction shows how this can be done.

\begin{defn}
Let $X$ be the space of $2^n$-tuples of continuous functions indexed
using the $2^n$ elements $\alpha \in Z^n$. Thus if functions
$f_\alpha \in \mathcal{C}(B)$ for all $\alpha\in Z^n$, then these
functions define an element of $X$, denoted $\{f_\alpha\}_{\alpha
\in Z^n} \in X$. Define the linear map $K : X \rightarrow
\mathcal{C}^1_\infty(B)$ as
\[
    K\left(\{f_\alpha\}_{\alpha
\in Z^n}\right)=\sum_{\alpha \in Z^n} G_\alpha f_\alpha
\]
where $G_\alpha:\mathcal{C}(B)\rightarrow \mathcal{C}^1_\infty(B)$
is given by
\[
    G_\alpha h=\left(\prod_{i=1}^n g_{i,\alpha_i}\right) h
\]
and where
\begin{align*}
    &(g_{i,j} h)(x_1,\dots,x_n) =\begin{cases}
    h(x_1,\ldots,x_{i-1},0,x_{i+1},\ldots,x_n) & j=0\\
    \int_{0}^{x_i} h(x_1,\ldots,x_{i-1},s,x_{i+1},\ldots,x_n) \,ds&
    j=1.
    \end{cases}
\end{align*}
\end{defn}

In practice, the functions $f_\alpha$ represent either the partial
derivatives of a function or approximations to those partial
derivatives. The following examples illustrate the construction.

\noindent\textbf{Example:} If $p=K(\{q_\alpha\}_{\alpha \in Z^2})$,
then
\begin{align*}
    p(x_1,x_2)&=\int_0^{x_1} \int_0^{x_2}
        q_{(1,1)}(s_1,s_2)\,ds_1\, ds_2\\
        &+\int_0^{x_1} q_{(1,0)}(s_1,0)\,ds_1\\
        &+\int_0^{x_2} q_{(0,1)}(0,s_2)\,ds_2\\
        &+q_{(0,0)}(0,0).
\end{align*}

Notice that this structure automatically gives a way of
approximating the partial derivatives of $p$. e.g.
\[
    \frac{\partial }{\partial x_1}p(x_1,x_2)=\int_0^{x_2}
        q_{(1,1)}(x_1,s_2)\, ds_2 + q_{(1,0)}(x_1,0).
\]

If $n=3$, then
\begin{align*}
    p(x_1,x_2,x_3)&=\int_0^{x_1} \int_0^{x_2} \int_0^{x_3}
        q_{(1,1,1)}(s_1,s_2,s_3)\,ds_1\, ds_2\, d s_3\\
    &+\int_0^{x_1} \int_0^{x_2}
        q_{(1,1,0)}(s_1,s_2,0)\,ds_1\, ds_2\\
    &+\int_0^{x_1} \int_0^{x_3}
        q_{(1,0,1)}(s_1,0,s_3)\,ds_1\, d s_3\\
    &+ \int_0^{x_2} \int_0^{x_3}
        q_{(0,1,1)}(0,s_2,s_3)\, ds_2\, d s_3\\
    &+\int_0^{x_1} q_{(1,0,0)}(s_1,0,0)\,ds_1\\
    &+\int_0^{x_2} q_{(0,1,0)}(0,s_2,0)\,ds_2\\
    &+\int_0^{x_3}q_{(0,0,1)}(0,0,s_3)\,d s_3\\
    &+q_{(0,0,0)}(0,0,0).
\end{align*}

The following lemma shows that when the $f_\alpha$ are the partial
derivatives of a function, we recover the original function. The
proof is simple and works by using a single integral identity to
repeatedly expand the function.

\begin{lem}\label{lem:lemma3}
For $v \in \mathcal{C}_\infty^1(B)$, if $f_\alpha=D^\alpha v$ for
all $\alpha \in Z^n$, then $K(\{f_\alpha\}_{\alpha \in Z^n})=v$.
\end{lem}

\begin{proof}
By assumption, $v \in \mathcal{C}_\infty^1(B)$ and so $D^\alpha v$
exist for all $\alpha \in Z^n$. The statement then follows from
application of the identity $v=g_{i,0} v + g_{i,1}
\frac{\partial}{\partial x_i} v$ recursively for $i=1,\ldots,n$.
That is, we make the sequential substitutions $v \mapsto g_{1,0} v +
g_{1,1} \frac{\partial}{\partial x_1} v$, $v \mapsto g_{2,0} v +
g_{2,1} \frac{\partial}{\partial x_1} v$ , $\dots$, $v \mapsto
g_{n,0} v + g_{n,1} \frac{\partial}{\partial x_n} v$. Then we have
the following expansion.
\begin{align*}
v=& g_{1,0}v+g_{1,1}D^{(1,0,\dots,0)}v\\
=&g_{1,0}g_{2,0}v+g_{1,0}g_{2,1}D^{(0,1,\dots,0)}v+g_{2,0}g_{1,1}D^{(1,0,\dots,0)}v\\
&+g_{2,1}g_{1,1}D^{(1,1,\dots,0)}v\\
=&g_{1,0}g_{2,0}g_{3,0}v+g_{1,0}g_{2,0}g_{3,1}D^{(0,0,1,0,\dots,0)}v\\
&+g_{1,0}g_{2,1}g_{3,0}D^{(0,1,0,\dots,0)}v+g_{1,0}g_{2,1}g_{3,1}D^{(0,1,1,0,\dots,0)}v\\
&+g_{2,0}g_{1,1}g_{3,0}D^{(1,0,\dots,0)}v+g_{2,0}g_{1,1}g_{3,1}D^{(1,0,1,0,\dots,0)}v\\
&+g_{2,1}g_{1,1}g_{3,0}D^{(1,1,\dots,0)}v+g_{2,1}g_{1,1}g_{3,1}D^{(1,1,1,0,\dots,0)}v\\
=& \cdots\\
=&K(\{D^\alpha v\}_{\alpha \in Z^n}),
\end{align*}
as desired.
\end{proof}

The following lemma states that the linear map $K$ is Lipschitz
continuous in an appropriate sense. This means that a small error in
the partial derivatives, $f_\alpha$, results in a small error of the
construction $K(\{f_\alpha\}_{\alpha \in Z^n})$ and all of its
partial derivatives.

\begin{lem}\label{lem:lemmatemp1}
Suppose $p=\{p_\alpha\}_{\alpha \in Z^n} \in X$ and
$q=\{q_\alpha\}_{\alpha \in Z^n} \in X$, then
\[
\max_{\beta \in Z^n}\norm{D^\beta K p - D^\beta Kq}_\infty \le 2^n
\max_{\alpha \in Z^n}\norm{p_\alpha - q_\alpha}_\infty.
\]
\end{lem}
\begin{proof}

This proof works by noticing that the map from any function
$f_\alpha$ to any of the partial derivatives of
$K(\{f_\alpha\}_{\alpha \in Z^n})$ is defined by the composition of
the operators $g_{i,j}$, all of which have small gain. To see this,
first note the following
\begin{align*}
\frac{\partial}{\partial x_i} g_{j,k} f =\begin{cases}g_{j,k}
\frac{\partial}{\partial x_i} f &i \neq
j\\f&i=j,k=1\\0&i=j,k=0.\end{cases}
\end{align*}
Then for all $\alpha, \beta \in Z^n$,
\begin{align*}
\frac{\partial^\beta}{\partial x^\beta}G_\alpha f=
    \begin{cases}
        0 & \alpha_i < \beta_i \text{ for some }i\\
        \displaystyle{\left(\prod_{\substack{i=1 \\ \beta_i \neq 1}}^n g_{i,\alpha_i}\right) f}& \text{otherwise.}
    \end{cases}
\end{align*}
Now, for any $f \in \mathcal{C}(B)$, it follows from the mean value
theorem that for any $x \in B$,
\begin{align*}
|(g_{i,1}f)(x)| &= \left|\int_{0}^{x_i}
f(x_1,\ldots,x_{i-1},\nu,x_{i+1},\ldots,x_n) \,d\nu\right|\\
& \le \sup_{s \in [-1,1]}
|f(x_1,\ldots,x_{i-1},s,x_{i+1},\ldots,x_n)|\\
&\le \|f\|_\infty.
\end{align*}
Thus
\[
    \norm{g_{i,1}f}_\infty \le \norm{f}_\infty
    \]
    for any $i$. Also, it
is clear that for any $i$,
\[
|(g_{i,0}f)(x)|=\left|f(x_1,\ldots,x_{i-1},0,x_{i+1},\ldots,x_n)\right|\le\|f\|_\infty.
\]
Therefore the $g_{i,j}$ have small gain, since
$\|g_{i,j}f\|_\infty\le \|f\|_\infty$ for any $i,j$. Now since for
any $\beta \in Z^n$,
\[
\frac{\partial^\beta}{\partial x^\beta}G_\alpha
\]
is the composition of $g_{i,j}$, induction can be used to prove the
the following for all $\alpha,\beta \in Z^n$.
\[
    \left\|\frac{\partial^\beta}{\partial x^\beta}G_\alpha f\right\| \le \norm{f}.
\]

Therefore
\begin{align*}
\norm{D^\beta Kp -D^\beta K q}_\infty&=\left\|\frac{\partial^\beta}{\partial x^\beta}K(p-q)\right\|_\infty \\
&\le \sum_{\alpha \in Z^n}\left\|\frac{\partial^\beta}{\partial x^\beta} G_\alpha (p_\alpha-q_\alpha)\right\|\\
&\le \sum_{\alpha \in Z^n}\norm{p_\alpha-q_\alpha}\\
&\le  2^n \max_{\alpha \in Z^n}\norm{p_\alpha-q_\alpha}
\end{align*}
for any $\beta \in Z^n$, as desired.
\end{proof}

The following theorem combines Lemmas~\ref{lem:lemma3}
and~\ref{lem:lemmatemp1} with the Weierstrass approximation theorem.
It says that the polynomials are dense in $\mathcal{C}_{\infty}^{1}$
with respect to the Sobolev norm for $W^{1,\infty}$, among others.
%

\begin{thm} \label{thm:theoremtemp1}Suppose $v \in \mathcal{C}_\infty^1(B)$. Then for any $\epsilon > 0$, there exists a polynomial
$p$, such that
\[
    \max_{\alpha \in Z^n}\norm{ D^\alpha p- D^\alpha v}_\infty \le \epsilon.
\]
\end{thm}
\begin{proof}
Since $v \in \mathcal{C}_\infty^1(B)$, $D^\alpha v \in
\mathcal{C}(B)$ for all $\alpha \in Z^n$. By the Weierstrass
approximation theorem, there exist polynomials $q_\alpha$ such that
\[
\max_{\alpha \in Z^n} \norm{q_\alpha - D^\alpha v}\le
\frac{\epsilon}{2^n}
\]
Let $q=\{q_\alpha\}_{\alpha \in Z^n}$ and $p=Kq$. Since the
$q_\alpha$ are polynomial, $p$ is polynomial. Let $f=\{D^\alpha
v\}_{\alpha \in Z^n}$. By Lemma~\ref{lem:lemma3}, $v=Kf$. Thus by
Lemma~\ref{lem:lemmatemp1}, we have that
\begin{align*}
    \max_{\alpha \in Z^n}\norm{D^\alpha p-D^\alpha v}&=\max_{\alpha \in Z^n} \norm{D^\alpha K q - D^\alpha K f} \\
    &\le 2^n \max_{\alpha \in Z^n} \norm{
    q_\alpha - D^\alpha v} \le \epsilon.
\end{align*}
\end{proof}

Theorem~\ref{thm:theoremtemp1} shows that for any continuously
differentiable function, $f$, there exists an arbitrarily good
polynomial approximation to the function, with error defined using
the norm $\max_{\alpha \in Z^n}\norm{ D^\alpha f}_\infty$. The proof
can be made constructive by using the Bernstein polynomials to
approximate the partial derivatives. If the partial derivatives are
Lipschitz continuous, then this method also gives explicit bounds on
the error. In practice, numerical experiments indicate that our
constructions, at least in 2 dimensions, tend to have error roughly
equivalent to the standard Bernstein polynomial approximations.



\section{Polynomial Lyapunov Functions}

In this section, we demonstrate that polynomial Lyapunov functions
can be used to approximate continuous Lyapunov functions. To be a
Lyapunov function, a polynomial approximation must satisfy certain
constraints. In particular, if $v$ is a Lyapunov function and $p$ is
a polynomial approximation to $v$, then $p$ is also a Lyapunov
function if it satisfies an error bound of the form
\[
    \left\|\frac{v(x)-p(x)}{x^T x}\right\|_\infty \le \epsilon.
\]
For $p$ to prove exponential stability, the derivatives of $p$ and
$v$ must satisfy a similar bound. The justification for this form of
the error bound is that the error should be everywhere bounded on a
compact set, but in addition must decay to zero near the origin.

The idea behind our proof of the existence of such a polynomial,
$p$, is to combine a Taylor series approximation with the
Weierstrass approximation theorem. Specifically, a second order
Taylor series expansion about a point, $x_0$, has the property that
the error, or residue, $R$, satisfies
\[
\frac{R(x,x_0)}{x^T x} \rightarrow 0
\]
as $x \rightarrow x_0$. However, the error in the Taylor series is
not guaranteed to converge uniformly over an arbitrary compact set
as the order of the expansion increases. The Weierstrass
approximation theorem, on the other hand, gives approximations which
converge uniformly on a compact set, but in general no Weierstrass
approximation will have the residual convergence property mentioned
above for any point, $x_0$. Our approach, then, is to use a second
order Taylor series expansion to guarantee accuracy near the origin.
We then use a Weierstrass polynomial approximation to the error
between the Taylor series and the function away from the origin to
cancel out this error and guarantee a uniform bound. We then use an
approach similar to that taken in Lemmas~\ref{lem:lemmatemp1}
and~\ref{thm:theoremtemp1} to show that the map $K$ can be used to
construct polynomial approximations to differentiable functions in
the norm
\[
    \max_{\alpha \in Z^n}\left\|\frac{D^\alpha v(x)}{x^T
    x}\right\|_\infty.
\]


%

We begin by combining the second order Taylor series expansion and
the Weierstrass approximation.

\begin{lem}\label{lem:lemma2}
Suppose $v \in \mathcal{C}_1^2(B)$. Then for any $\epsilon
> 0$, there exists a polynomial $p$ such that
\[
    \left\|\frac{p(x)-v(x)}{x^T x}\right\|_\infty \le \epsilon.
 \]
\end{lem}
\vspace{.5cm}

\begin{proof}
Let the polynomial $m$ be defined using the second order Taylor
series expansion for $v$ about $x=0$ as
\[
    m(x)=v(0)+\sum_{i=1}^n x_i \frac{\partial v}{\partial
x_i }(0)+\frac{1}{2}\sum_{i,j=1}^n x_i x_j \frac{\partial^2
v}{\partial x_i \partial x_j}(0).
\]
Then $m$ approximates $v$ near the origin and specifically
\[
    (v-m)(0)=\frac{\partial (v-m)}{\partial x_i
}(0)=\frac{\partial^2 (v-m)}{\partial x_i \partial x_j}(0)=0
\]
for $i,j=1,\dots,n$.

Now define
\[
    h(x)=\begin{cases}0 & x=0\\
    \frac{v(x)-m(x)}{x^T x}&\text{otherwise}.\end{cases}
\]
Then from Taylor's theorem(See, e.g.~\cite{marsden_veccal}), we have
that
\begin{align*}
    v(x)=&v(0)+\sum_{i=1}^n x_i \frac{\partial v}{\partial
x_i }(0)+\frac{1}{2}\sum_{i,j=1}^n x_i x_j \frac{\partial^2
v}{\partial x_i \partial x_j}(0)+R_2(x)
\end{align*}
where $\frac{R_2(x)}{x^T x}\rightarrow 0$ as $x \rightarrow 0$.
Therefore
\[
    h(x)=\frac{v(x)-m(x)}{x^T x}=\frac{R_2(x)}{x^T x} \rightarrow 0
\]
as $x\rightarrow 0$ and so $h(x)$ is continuous at $0$. Since
$v(x)-m(x)$ and $x^T x$ are continuous and $x^T x \neq 0$ on every
domain not containing $x = 0$ and every point $x \neq 0$ has a
neighborhood not containing $x=0$, we conclude that $h(x)$ is
continuous at every point $x \in \R^n$.


We can now use the Weierstrass approximation theorem, which states
that there exists some polynomial $q$ such that
\[
    \norm{q-h}_{\infty}\le \epsilon.
\]
The Taylor and Weierstrass approximations are now combined as
$p(x)=m(x)+q(x)x^T x$. Then $p$ is polynomial and
\begin{align*}
\left\|\frac{p(x)-v(x)}{x^T x} \right\|_\infty
&=\left\|\frac{m(x)+q(x)x^T
x-v(x)}{x^T x} \right\|_\infty \\
&=\left\|\frac{m(x)-v(x)}{x^T x}+h(x)+\left(q(x)-h(x)\right) \right\|_\infty\\
&=\left\|q(x)-h(x) \right\|_\infty \le \epsilon
\end{align*}

\end{proof}

The proof of the following lemma closely follows that of
Lemma~\ref{lem:lemmatemp1}. However, the presence of the $1/x^T x$
term poses significant technical challenges. In particular, small
gain of the operators $g_{i,j}$ is no longer sufficient. We instead
use an inductive reasoning, similar to small gain, which is
described in the proof.

\begin{lem}\label{lem:lemma4}
Let $p=\{p_\alpha\}_{\alpha \in Z^n} \in X$ and
$q=\{q_\alpha\}_{\alpha \in Z^n} \in X$. Then
\[
\max_{\beta \in Z^n}\left\|\frac{D^\beta Kp(x) -D^\beta Kq(x)}{x^T
x}\right\|_\infty\le 2^n \max_{\alpha \in Z}\left\|\frac{p_\alpha(x)
- q_\alpha(x)}{x^T x}\right\|_\infty.
\]
\end{lem}

\begin{proof}
Recall that from the definition of $g_{j,k}$, we have that
\begin{align*}
\frac{\partial}{\partial x_i} g_{j,k} f =\begin{cases}g_{j,k}
\frac{\partial}{\partial x_i} f &i \neq
j\\f&i=j,k=1\\0&i=j,k=0\end{cases},
\end{align*}
which implies
\begin{align*}
\frac{\partial^\beta}{\partial x^\beta}G_\alpha f=
    \begin{cases}
        0 & \alpha_i < \beta_i \text{ for some }i\\
        \displaystyle{\left(\prod_{\substack{i=1\\\beta_i \neq 1}}^n g_{i,\alpha_i}\right) f}& \text{otherwise.}
    \end{cases}
\end{align*}

Now consider the term
\[
    \frac{1}{x^T x}(g_{i,j}f)(x).
\]

We would like to obtain bounds on this function. For $j=1$, and for
any $x \in B$,
\begin{align*}
    &\left|\frac{1}{x^T x}(g_{i,1}f)(x)\right|\\
    &=\left|\int_{0}^{x_i}\frac{ f(x_1,\ldots,x_{i-1},t,x_{i+1},\ldots,x_n)}{\sum_{k =1}^n x_k^2} \,d t \right|\\
    &\le\sup_{\nu \in [-|x_i|,|x_i|]}\frac{ \left|f(x_1,\ldots,x_{i-1},\nu,x_{i+1},\ldots,x_n)\right|}{\sum_{k =1}^n x_k^2} \\
&\le \sup_{\nu \in
[-|x_i|,|x_i|]}\frac{\left|f(\ldots,x_{i-1},\nu,x_{i+1},\ldots)\right|}{\nu^2+\sum_{k
\neq i}^n x_k^2} \\
&\le \left\|\frac{f(s)}{s^T s}\right\|_\infty .
\end{align*}
Here the first inequality is due to the mean value theorem and that
$|x_i|\le 1$ and the second inequality follows since $x_i^2 \ge
\nu^2$ for $\nu \in [-|x_i|,|x_i|]$. Therefore, we have
\[
\left\|\frac{1}{x^T x}(g_{i,1}f)(x)\right\|_\infty \le
\left\|\frac{1}{x^T x}f(x)\right\|_\infty.
\]
Similarly, if $j=0$, then for any $x \in B$,
\begin{align*}
    &\left|\frac{1}{x^T x}(g_{i,0}f)(x)\right|\\
    &=\left|\frac{f(x_1,\ldots,x_{i-1},0,x_{i+1},\ldots,x_n)}{x^T x} \right|\\
&\le \left|\frac{f(x_1,\ldots,x_{i-1},0,x_{i+1},\ldots,x_n)}{\sum_{k \neq i}^n x_k^2}\right|\\
&\le \left\|\frac{f(s)}{s^T s} \right\|_\infty,
\end{align*}
where the first inequality follows since $x_i^2 \ge 0$. Therefore,
we have that for $j\in \{0,1\}$ and $i = 1 ,\dots n$,
\[
\left\|\frac{1}{x^T x}(g_{i,j}f)(x)\right\|_\infty \le
\left\|\frac{1}{x^T x}f(x)\right\|_\infty.
\]

Since the terms $G_\alpha$ are compositions of the $g_{i,j}$, we can
apply the above bounds inductively. Specifically, we see that for
any $\beta \in Z^n$,
\begin{align*}
&\left\|\frac{1}{x^T x}\left(\frac{\partial^\beta}{\partial
x^\beta}G_\alpha f\right)(x)\right\|_\infty\\
&= \left\|\frac{1}{x^T x}\left(\left(\prod_{\substack{i=1\\\beta_i \neq 1}}^n g_{i,\alpha_i}\right) f\right)(x)\right\|_\infty\\
&\le \left\|\frac{1}{x^T x}\left(\left(\prod_{\substack{i=2\\\beta_i \neq 1}}^n g_{i,\alpha_i}\right) f\right)(x)\right\|_\infty\\
&\cdots
\le \left\|\frac{f(x)}{x^T x} \right\|_\infty.
\end{align*}

 Now that we have bounds on the $G_\alpha$, we can use the triangle inequality to deduce that for any $\beta \in Z^n$,
\begin{align*}
&\left\|\frac{D^\beta K p(x)-D^\beta K q (x)}{x^T x}\right\|_\infty\\
&=\left\|\frac{1}{x^T x}\frac{\partial^\beta}{\partial x^\beta}K(p-q)(x)\right\|_\infty \\
&\le \sum_{\alpha \in Z^n}\left\|\frac{1}{x^T x}\left(\frac{\partial^\beta}{\partial x^\beta} G_\alpha (p_\alpha-q_\alpha)\right)(x)\right\|_\infty\\
&\le \sum_{\alpha \in Z^n}\left\|\frac{p_\alpha(x)-q_\alpha(x)}{x^T x}\right\|_\infty\\
&\le 2^n \max_{\alpha \in
Z^n}\left\|\frac{p_\alpha(x)-q_\alpha(x)}{x^T x}\right\|_\infty.
\end{align*}
\end{proof}

The following theorem gives the main approximation result of the
paper. It combines Lemmas~\ref{lem:lemma2} and~\ref{lem:lemma4} to
show that polynomials are dense in the space $\mathcal{C}_1^{n+2}$
with respect to the weighted $W^{1,\infty}$ norm with weight $1/x^T
x$, among others.

\begin{thm}\label{thm:theorem1}
Suppose $v$ is a function with partial derivatives
\[
    D^\alpha v \in \mathcal{C}_1^2(B)
\]
for all $\alpha \in Z^n$. Then for any $\epsilon
> 0$, there exists a polynomial $p$, such that
\[
\max_{\alpha \in Z^n} \left\| \frac{D^\alpha p(x) - D^\alpha
v(x)}{x^T x}\right\|_\infty \le \epsilon.
\]
\end{thm}

\begin{proof}
The proof is similar to that for Theorem~\ref{thm:theoremtemp1}. By
assumption, $D^\alpha v \in \mathcal{C}_1^2(B)$ for all $\alpha \in
Z^n$. By Lemma~\ref{lem:lemma2}, there exist polynomial functions
$r_\alpha$ such that
\[
    \max_{\alpha \in Z^n}\left\|\frac{r_\alpha(x) - D^\alpha v(x)}{x^T x}\right\|_\infty \le
\frac{\epsilon}{2^n}.
\]
Let $r=\{r_\alpha\}_{\alpha \in Z^n}$ and $p=Kr$. Then $p$ is
polynomial since the $r_\alpha$ are polynomial. Let $h=\{D^\alpha
v\}_{\alpha\in Z^n}$. Then by Lemma~\ref{lem:lemma3}, $v=Kh$.
Therefore by Lemma~\ref{lem:lemma4}, we have

\begin{align*}
&\max_{\alpha \in Z^n} \left\| \frac{D^\alpha p (x) - D^\alpha v
(x)}{x^T x}\right\|_\infty \\
&=\max_{\alpha \in Z^n} \left\| \frac{D^\alpha K r (x) - D^\alpha K
h (x)}{x^T
x}\right\|_\infty\\
&\le 2^n \max_{\alpha \in Z}\left\|\frac{r_\alpha(x) - D^\alpha
v(x)}{x^T x}\right\|_\infty \le \epsilon,
\end{align*}
as desired.
\end{proof}

We now conclude the section by using Theorem~\ref{thm:theorem1} to
show that the existence of a sufficiently smooth Lyapunov function
which proves exponential stability on a bounded set implies the
existence of a polynomial Lyapunov function which proves exponential
stability on the set.

\begin{prop}\label{prop:prop1}Let $\Omega \subset \R^n$ be bounded
with radius $r$ in norm $\norm{\cdot}_\infty$ and $f(x)$ be
uniformly bounded on $B_r:=\{x \in \R^n: \norm{x}_\infty \le r\}$.
Suppose there exists a $v:B_r \rightarrow \R$ with $D^\alpha v \in
\mathcal{C}_1^2(B_r)$ for all $\alpha\in Z^n$ and such that
\begin{align*}
\beta_0 \norm{x}^2 \le v(x)\le \gamma_0 \norm{x}^2\\
\nabla v(x)^T f(x) \le -\delta_0 \norm{x}^2,
\end{align*}
for some $\beta_0>0$, $\gamma_0>0$ and $\delta_0 > 0$ and all $x \in
\Omega$. Then for any $\beta<\beta_0$, $\gamma > \gamma_0$ and
$\delta<\delta_0 $ there exists a polynomial $p$ such that
\begin{align*}
\beta \norm{x}^2 \le p(x)\le \gamma \norm{x}^2\\
\nabla p(x)^T f(x) \le -\delta \norm{x}^2
\end{align*}
for all $x \in \Omega$.
\end{prop}

\begin{proof}
Let $\hat{v}(x)=v(rx)$ and
\[
b = \norm{f}_\infty=\sup_{\norm{x}_\infty \le r} \norm{f(x)}_\infty.
\]
Choose
$0<d<\min\{\beta_0-\beta,\gamma-\gamma_0,\frac{\delta_0-\delta}{n
b}\}$. By Theorem~\ref{thm:theorem1}, there exists a polynomial,
$\hat p$, such that for $\norm{x}_\infty \le 1$,
\[
\left|\frac{\hat p(x)-\hat v(x)}{x^T x}\right| \le \frac{d}{r^2}
\]
and
\[
\left| \frac{\frac{\partial \hat p }{\partial x^i}(x) -
\frac{\partial \hat v }{\partial x^i}(x)}{x^T x}\right| \le
\frac{d}{r^2}
\]
for $i=1,\dots,n$. Now let $p(x)=\hat p(x/r)$. Then for $x \in
\Omega$, $\norm{x}_\infty \le r$ and so $\norm{x/r}_\infty\le 1$.
Therefore we have the following for all $x \in \Omega$,
\begin{align*}
    p(x)&=v(x)+\hat p (x/r) - \hat v(x/r)\\
    &=v(x)+\frac{\hat p(x/r)-\hat v(x/r)}{(x/r)^T(x/r)} r^2 x^T x\\
    &\ge (\beta_0-d)x^T
    x\\
    &\ge \beta x^T x.
\end{align*}
Likewise, 
\begin{align*}
    p(x)&=v(x)+\frac{\hat p(x/r)-\hat v(x/r)}{(x/r)^T(x/r)} r^2 x^T x\\
    &\le (\gamma_0+d)x^T x\\
    &\le \gamma x^T x.
\end{align*}
Finally, 
\begin{align*}
\nabla p(x)^T f(x)&
= \frac{\nabla(\hat p(x/r)-\hat v(x/r))^T f}{x^T x}x^T x +\nabla v(x)^T f(x)\\
&= \sum_{i=1}^n \left(r^2 \frac{\frac{\partial \hat p}{\partial
x_i}(x/r)-\frac{\partial \hat v}{\partial x_i}(x/r)}{(x/r)^T
(x/r)}f_i(x)  \right)x^T x
+\nabla v(x)^T f(x)\\
&\le n\, d\, b\, x^T x - \delta_0 x^T x\\
&\le - \delta x^T x.
\end{align*}
Thus the proposition holds for $x \in \Omega$.
\end{proof}

A consequence of Proposition~\ref{prop:prop1} is that when
estimating exponential rates of decay, using polynomial Lyapunov
functions does not result in a reduction of accuracy. i.e. if there
exists a continuous Lyapunov function proving an exponential rate of
decay with bound $\alpha_0$, then for any $0 < \alpha < \alpha_0$,
there exists a polynomial Lyapunov function which proves an
exponential rate of decay with bound $\alpha$.

\section{Lyapunov Stability}\label{sec:lyapunov}

Consider the system
\begin{equation}
    \dot{x}(t)=f(x(t))
    \label{eqn:nonlinearODE2}
\end{equation}
where $f : \R^n \rightarrow \R^n$, $f(0)=0$ and $x(0)=x_0$. We
assume that there exists an $r \ge 0$ such that for any
$\norm{x_0}_\infty \le r$, Equation~\eqref{eqn:nonlinearODE2} has a
unique solution for all $t\ge 0$. We define the solution map $A:
\R^n \rightarrow \mathcal{C}([0,\infty))$ by
\[
(Ay)(t)=x(t)
\]
for $t\ge 0$, where $x$ is the unique solution of
Equation~\eqref{eqn:nonlinearODE2} with initial condition $y$. The
following comes from~\citet{vidyasagar_book}.

\begin{thm}\label{thm:converse}
Consider the system defined by Equation~\eqref{eqn:nonlinearODE2}
and suppose that $f \in \mathcal{C}_1^k(\R^n)$ for some integer $k
\ge 1$. Suppose that there exist constants $\mu,\delta,r>0$ such
that
\[
    \norm{(A x_0)(t)}_2 \le \mu \norm{x_0}_2e^{-\delta t}
\]
for all $t \ge 0$ and $\norm{x_0}_2 \le r$. Then there exists a
$\mathcal{C}_1^k(\R^n)$ function $V:\R^n \rightarrow \R$ and
constants $\alpha,\beta,\gamma,\mu > 0$ such that
\begin{align*}
\alpha \norm{x}_2^2 \le &V(x) \le \beta \norm{x}_2^2\\
\frac{\partial }{\partial t}V((Ax)(t))\le &-\gamma \norm{x}_2^2
\end{align*}
for all $\norm{x}_2 \le r$.
\end{thm}

The following gives a converse Lyapunov result which may be taken as
the main conclusion of the paper.\\

\begin{thm}\label{prop:conversepoly}
Consider the system defined by Equation~\eqref{eqn:nonlinearODE2}
where $f \in C_1^{n+2}(\R^n)$ . Suppose there exist constants
$\mu,\delta,r>0$ such that
\[
    \norm{A x_0(t)}_2 \le \mu \norm{x_0}_2e^{-\delta t}
\]
for all $t \ge 0$ and $\norm{x_0}_2 \le r$.

Then there exists a \eemph{polynomial} $v:\R^n \rightarrow \R$ and
constants $\alpha,\beta,\gamma,\mu > 0$ such that
\begin{align*}
\alpha \norm{x}_2^2 \le v(x) \le &\beta \norm{x}_2^2\\
\nabla v(x)^T f(x)\le &-\gamma \norm{x}_2^2
\end{align*}
for all $\norm{x}_2 \le r$.
\end{thm}

\begin{proof}
We use Theorem~\ref{thm:converse} to prove the existence of a
Lyapunov function $V : \R^n \rightarrow \R$ with $ V \in
\mathcal{C}_1^{n+2}(\R^n)$ satisfying the conditions on $\Omega :=
\{x:\norm{x}_2 \le r\}$. Since $\mathcal{C}_\infty^1(\R^n) \subset
\mathcal{C}_1^n(\R^n)$, Theorem~\ref{prop:prop1} proves the
existence of a polynomial function $v$ which satisfies the theorem
statement.
\end{proof}

An important corollary of Theorem~\ref{prop:conversepoly} is that
ordinary differential equations defined by polynomials have
polynomial Lyapunov functions. Since polynomial optimization is
typically applied to systems defined by polynomials, this means that
the assumption of a polynomial Lyapunov function is not
conservative.

In polynomial optimization, it is common to use Positivstellensatz
results to find locally positive polynomial Lyapunov functions in a
manner similar to the $S$-procedure. When the polynomial $v$ can be
assumed to be positive, i.e. $v(x)>0$ for all $x$, these conditions
are necessary and sufficient.
See~\citet{stengle_1973},~\citet{schmudgen_1991},
and~\citet{putinar_1993} for strong theoretical contributions.
Unfortunately, the polynomial Lyapunov functions are not positive
since $v(0)=0$, and so these conditions are no longer necessary and
sufficient. However, Positivstellensatz results still allow us to
search over polynomial Lyapunov functions in a manner which has
proven very effective in practice.

\begin{defn}
A polynomial, $p$, is \eemph{sum-of-squares}, if there exists a
$K>0$ and polynomials $g_i$ for $i=1,\dots,K$ such that
\[
    p(x)=\sum_{i=1}^K g_i(x)^2.
\]
\end{defn}

\begin{prop}
Consider the system defined by Equation~\eqref{eqn:nonlinearODE2}
where $f$ is polynomial. Suppose there exists a polynomial $v:\R^n
\rightarrow \R$, a constant $\epsilon>0$, and sum-of-squares
polynomials $s_1,s_2,t_1,t_2:\R^n \rightarrow \R$ such that
\[
    v(x)-s_1(x)(r-x^T x) - s_2(s)- \epsilon \, x^T x=0
\]
and
\[
    -\nabla v(x)^T f(x) - t_1(x)(r-x^T x) -t_2(x) -\epsilon \, x^T x= 0
\]
Then there exist constants $\mu,\delta,r>0$ such that
\[
    \norm{(A x_0)(t)}_2 \le \mu \norm{x_0}_2e^{-\delta t}
\]
for all $t \ge 0$ and $x_0 \in Y(v,r)$ where $Y(v,r)$ is the largest
sublevel set of $v$ contained in the ball $\norm{x}^2 \le r$.
\end{prop}

See~\citet{papachristodoulou_2002} for a proof and more details on
using semidefinite programming to construct solutions to this
polynomial optimization problem.

\section{Conclusion}
The main result of this paper is a proof that exponential stability
of a sufficiently smooth nonlinear ordinary differential equation on
a bounded region implies the existence of a polynomial Lyapunov
function which decreases exponentially on the region. A corollary of
this result is that ordinary differential equations defined by
polynomials have polynomial Lyapunov functions. An important
application of polynomial programming is the search for a polynomial
Lyapunov function which proves local exponential stability. Our
results, therefore, tend to support continued research into
improving polynomial optimization algorithms.

In addition, as a byproduct of our proof, we were able to give a
method for constructing polynomial approximations to differentiable
functions. The interesting feature of this construction is the
guaranteed convergence of the derivatives of the approximation.
Another consequence of the results of this paper is that the
polynomials are dense in $\mathcal{C}_\infty^1(B)$ with respect to
the Sobolev norm $W^{1,\infty}(B)$.

\bibliography{weierstrass}

\end{document}